\documentclass[11pt,letterpaper]{amsart}
\usepackage[letterpaper,hmargin=1.25in,vmargin=1.5in]{geometry}
\usepackage{latexsym, amssymb, amsmath}
\usepackage{booktabs}   
\usepackage{multirow}

\newtheorem{theorem}{Theorem}[section]
\newtheorem{lemma}[theorem]{Lemma}

\newtheorem{proposition}[theorem]{Proposition}

\theoremstyle{definition}

\makeatletter
    
    \@addtoreset{equation}{section}
  \makeatother

\begin{document}

\title[Complete quasi-Yamabe gradient solitons with bounded scalar curvature]
{Complete quasi-Yamabe gradient solitons with bounded scalar curvature}

\author{Shun Maeta}
\address{Department of Mathematics, Chiba University, 1-33, Yayoicho, Inage, Chiba, 263-8522, Japan.}
\curraddr{}
\email{shun.maeta@faculty.gs.chiba-u.jp~{\em or}~shun.maeta@gmail.com}
\thanks{The author is partially supported by the Grant-in-Aid for Scientific Research (C), No.23K03107, Japan Society for the Promotion of Science.}

\date{}

\dedicatory{}


\commby{}

\begin{abstract}
In this paper, we classify complete, nontrivial shrinking and steady quasi-Yamabe gradient solitons whose scalar curvature is bounded below by the soliton constant. We also classify complete, nontrivial expanding and steady quasi-Yamabe gradient solitons whose scalar curvature is bounded above by the soliton constant. 
\end{abstract}

\maketitle
\bibliographystyle{amsalpha}


\section{Introduction}\label{intro}

To analyze the structure of Riemannian manifolds, it is important to study geometric flows (cf. \cite{BS08}, \cite{Perelman02}).
One of the most interesting geometric flows is the Yamabe flow (cf. \cite{Hamilton89}).
Research on the Yamabe flow has developed rapidly in recent years (cf. \cite{Brendle05}, \cite{Brendle07}, \cite{Chow92},  \cite{SS03}, \cite{Ye94}).
To understand the Yamabe flow, it is important to study singularity models of the flow.
Yamabe solitons are special solutions to the Yamabe flow that are expected to serve as singularity models. Research on Yamabe solitons has also progressed rapidly in recent years (cf. \cite{BR13}, \cite{CSZ12}, \cite{CMM12}, \cite{CMMR17}, \cite{CD08}, \cite{DS13}, \cite{Maeta21}, \cite{Maeta23}, \cite{Maeta24}).
With the aim of viewing Yamabe solitons in a more general framework, G. Huang and H. Li introduced the notion of a quasi-Yamabe gradient soliton \cite{HL14} and showed that any compact quasi-Yamabe gradient soliton is trivial. Tokura, Batista, Kai, and Barboza \cite{TBKB22} also showed a triviality result in more general settings. 
Therefore, our next problem is to study complete quasi-Yamabe gradient solitons.

Wang \cite{Wang13} provided an example of a noncompact quasi-Yamabe gradient soliton.
Catino, Mastrolia, Monticelli, and Rigoli \cite{CMMR17} provided a structure theorem for complete quasi-Yamabe gradient solitons and proved a rigidity result under nonnegative Ricci curvature. Gomes and Tokura \cite{GT25} showed rigidity results in more general settings.
Recently, Francisco, Filho and Fernandes \cite{FFF24} classified complete quasi-Yamabe gradient solitons assuming constant scalar curvature. 
On K\"ahler manifolds, the author completely classified quasi-Yamabe gradient solitons (cf. \cite{Maeta25}). 

In this paper, we classify complete quasi-Yamabe gradient solitons under very weak assumptions on the scalar curvature.


\section{Preliminaries}
In this section, we provide the definition of a quasi-Yamabe gradient soliton and a structure theorem for complete quasi-Yamabe gradient solitons by Catino, Mastrolia, Monticelli and Rigoli.

Let $(M,g)$ be a Riemannian manifold. 
If there exists a smooth function $F\in C^{\infty}(M)$ and constants $\lambda,c\in\mathbb{R}$ with $c\neq0$ such that
\begin{equation}\label{GQYS}
\nabla \nabla F=(R-\lambda)g+c\nabla F\nabla F
\end{equation}
where $\nabla\nabla F$ is the Hessian of $F$, $R$ is the scalar curvature of $M$, and $\nabla F$ is the gradient of $F$. 
If $\lambda>0$, $\lambda=0$, or $\lambda<0$, then the quasi-Yamabe gradient soliton is called shrinking, steady, or expanding.
The function $F$ is called the potential function.
If $F$ is constant, then the soliton is called {\em trivial}.

Catino, Mastrolia, Monticelli and Rigoli provided the structure theorem for
complete quasi-Yamabe gradient solitons. 

\begin{theorem}[\cite{CMMR17}]\label{structure}
Let $(M,g,F,c)$ be an $n$-dimensional complete quasi-Yamabe gradient soliton. Then $(M,g,F,c)$ is either

$(1)$ $M=\mathbb{R}\times N$ and $g=dr^2+\psi^2(r)g_N$, and $F$ has no critical point, or 

$(2)$ $M=[0,+\infty)\times \mathbb{S}^{n-1}$ and $g=dr^2+\psi^2(r)g_S,$ $r\in[0,+\infty)$, and $F$ has only one critical point $F'(0)=0$.

Here $\psi(r)=F'(r)e^{-cF(r)}$ is a positive smooth function on $M$ that depends only on $r$, and $g_S$ denotes the round metric of $\mathbb{S}^{n-1}$.
\end{theorem}

By differentiating $\psi$ by $r$, we have
\[
\psi'
=(F''-cF'^2)e^{-cF}=(R-\lambda)e^{-cF},
\]
\[
\psi''=(F'''-3cF'F''+c^2F'^3)e^{-cF}.
\]


\section{Triviality and results for complete quasi-Yamabe gradient solitons of all types}

In this section, we provide some results for complete quasi-Yamabe gradient solitons of the shrinking, steady, and expanding types.

By Theorem \ref{structure}, we need to consider warped product manifolds. 
By a well-known curvature formula for warped product manifolds, we have

\begin{equation}\label{R}
R\psi^2+(n-1)(n-2)\psi'^2+2(n-1)\psi\psi''=\bar R,
\end{equation}
where $R$ is the scalar curvature of $(M,g)$, and $\bar R$ is the scalar curvature of $N$ in case (1) of Theorem \ref{structure} and of $\mathbb{S}^{n-1}$ in case (2) of Theorem \ref{structure}. 

By \eqref{R}, we have 
\begin{equation}\label{R2}
\lambda\psi^2+\psi'\psi^2e^{cF}+(n-1)(n-2)\psi'^2+2(n-1)\psi\psi''=\bar R.
\end{equation}
Since $\psi$ and $F$ depend only on $r$, $\bar R$ is constant.
By differentiating both sides of \eqref{R2}, we have
\begin{align}\label{R3}
2\lambda \psi\psi'+\psi^2\psi''e^{cF}+2\psi\psi'^2e^{cF}+c\psi^3\psi'e^{2cF}
+2(n-1)^2\psi'\psi''+2(n-1)\psi\psi'''=0.
\end{align}

First, we consider when the warped product is the direct product, that is, when $\psi$ is constant. 
Interestingly, there exist no nontrivial complete quasi-Yamabe gradient solitons of direct-product type.

\begin{lemma}\label{constpsi}
Any $n$-dimensional $(n\geq3)$ complete quasi-Yamabe gradient soliton with constant $\psi$ is trivial.
\end{lemma}

\begin{proof}
By the assumption, the function $\psi>0$ is constant.
Set $\psi=F'e^{-cF}=a.$ We have
\[
F'(r)=
-\frac{a}{acr+c_1},
\]
where $c_1\in\mathbb{R}$.

Case (1) of Theorem \ref{structure}. Since $F'$ is not smooth on $\mathbb{R}$, this is a contradiction.  

Case (2) of Theorem \ref{structure}. Since $F'(0)=-\frac{a}{c_1}=0$, we have a contradiction.

\end{proof}

Secondly, we show that the scalar curvature of $N$ is positive when the scalar curvature of $M$ is non-negative.

\begin{lemma}\label{rgeq0}
Any $n$-dimensional $(n\geq3)$ complete nontrivial quasi-Yamabe gradient soliton with $R\geq0$ has $\bar R>0$.
\end{lemma}

\begin{proof}
Assume that $\bar R<0$. By \eqref{R}, we have $\psi''<0$.
Thus, $\psi$ is positive and concave.
Therefore, in Case (1) of Theorem \ref{structure}, $\psi$ is constant on $\mathbb{R}$, which cannot occur because of Lemma \ref{constpsi}.
Hence, one has $\bar R\geq0$. Assume that $\bar R=0$. 
By \eqref{R} again, one has $\psi''\leq0$. 
If $\psi''<0$ at least at one point, then $\psi$ is constant. By Lemma \ref{constpsi}, we have a contradiction.
Assume that $\psi''=0$ on $\mathbb{R}$, then we can assume that $\psi'=b\in\mathbb{R}$. By \eqref{R}, we have
\[
R\psi^2+(n-1)(n-2)b^2=0.
\]
By the equation, we have $b=0$. By Lemma \ref{constpsi} again,  we have a contradiction.

Case (2) of Theorem \ref{structure}. It is obvious.
\end{proof}

Thirdly, we consider the case $R>\lambda+\varepsilon$ for some $\varepsilon>0$.

%

\begin{proposition}\label{r+v}
Any complete nontrivial quasi-Yamabe gradient soliton with $R>\lambda+\varepsilon$ for some $\varepsilon>0$ and $c>0$ is rotationally symmetric.
More precisely, $M=[0,+\infty)\times \mathbb{S}^{n-1}$ and $g=dr^2+\psi^2(r)g_S,$ $r\in[0,+\infty)$.
\end{proposition}

\begin{proof}
We only need to consider (1) of Theorem \ref{structure}. Set $s=-r$. Then we have 
\[
-\dot\psi e^{cF}=R-\lambda>\varepsilon.
\]
Hence, one has $\dot\psi(s)<-\varepsilon e^{-cF(s)}$.
Since $\dot F(s)<0$, $F(s)$ is monotone decreasing, and $e^{-cF}$ is monotone increasing. Hence, on some interval $(s_0,+\infty)$, $e^{-cF}$ is bounded from below by some positive constant. Hence, $\dot\psi<-\alpha^2$ for some constant $\alpha$ on $(s_0,+\infty)$.
Since $\psi>0$, we have a contradiction.
\end{proof}

Finally, we consider the case $R<\lambda-\varepsilon$ for some $\varepsilon>0$.


\begin{proposition}\label{r-v}
Any $n$-dimensional $(n\geq3)$ complete quasi-Yamabe gradient soliton with $R<\lambda-\varepsilon$  for some $\varepsilon>0$ and $c<0$ is trivial.
\end{proposition}

\begin{proof}
By the assumption, we have 
\[
\psi' e^{cF}=R-\lambda<-\varepsilon.
\]
Hence, one has $\psi'(r)<-\varepsilon e^{-cF(r)}$.
Since $F'(r)>0$, $F(r)$ is monotone increasing, and $e^{-cF}$ is also monotone increasing. Hence, on some interval $(r_0,+\infty)$, $e^{-cF}$ is bounded from below by some positive constant. Hence, $\psi'<-\alpha^2$ for some constant $\alpha$ on $(r_0,+\infty)$.
Since $\psi>0$, we have a contradiction.
\end{proof}



\section{Shrinking and Steady quasi-Yamabe gradient solitons with $R>\lambda$}\label{shrste}



In this section, we classify complete shrinking and steady quasi-Yamabe gradient solitons with $R>\lambda$.

\begin{theorem}\label{nnp}
Let $(M,g,F,c)$ be an $n$-dimensional $(n\geq3)$  nontrivial complete shrinking or steady quasi-Yamabe gradient soliton with $c>0$. If $R>\lambda$, then $M$ is rotationally symmetric. More precisely, 
$M=[0,+\infty)\times \mathbb{S}^{n-1}$ and $g=dr^2+(F'(r)e^{-cF(r)})^2g_S,$ $r\in[0,+\infty)$, and $F$ has only one critical point $F'(0)=0$.

\end{theorem}

\begin{proof}

Assume that $\psi''>0$ at some point $r_0\in\mathbb{R}$, that is $\psi''>0$ on some open set $\Omega=(r_1,r_2)\ni r_0$.
If $\psi''(r_2)=0$, then by \eqref{R3}, we have
$\psi'''(r_2)<0.$ Hence, the positive smooth function $\psi'$ is monotone decreasing on $(r_2,r_3)$ for some $r_3$. By the same argument, we have that $\psi'$ is weakly monotone decreasing on $(r_2,+\infty)$.
On the other hand, if $r_2=+\infty$, and $\psi''(r_1)=0$, then the same argument shows that $\psi'''<0$ at $r_1$, which is a contradiction. Hence in this case, we have $\psi''>0$ on $\mathbb{R}$. However, the left hand side of \eqref{R2} goes to infinity as $r\nearrow+\infty$, which cannot occur. Therefore, we have $\psi'$ is weakly monotone decreasing on $(r_2,+\infty)$, and $\Omega$ must be $(-\infty,r_2)$. Note that $r_2$ is the maximum point of $\psi'$. Since $\psi'>0$ on $\mathbb{R}$, there exists $r_{-1}\in\Omega$ such that $\psi''(r_{-1})>0$ and $\psi'''(r_{-1})=0$. At $r_{-1}$, the left hand side of \eqref{R3} is positive. Hence, we have a contradiction, and $\Omega$ is empty. Therefore, we have $\psi''\leq0$ on $\mathbb{R}$. Since $\psi$ is a positive smooth function, by the same argument as in the proof of Lemma \ref{rgeq0}, $\psi$ is constant, and $M$ is trivial.

\end{proof}

In particular, we showed that any nontrivial complete steady quasi-Yamabe gradient soliton with $c>0$ is rotationally symmetric, which is a problem related to Perelman’s conjecture (cf. \cite{Perelman02}).


\section{Expanding and Steady quasi-Yamabe gradient solitons with $R<\lambda$}\label{exp}

In this section, we classify complete expanding and steady quasi-Yamabe gradient solitons with $R<\lambda$.

\begin{proposition}\label{e,r<l,c<0}
Let $(M,g,F,c)$ be an $n$-dimensional $(n\geq3)$ complete nontrivial expanding quasi-Yamabe gradient soliton with $R<\lambda$ and $c<0$, then $M$ is either

\begin{enumerate}
\item
the warped product 
\[
(\mathbb{R},dr^2)\times_{\psi}(N^{n-1}(0),\bar g),
\]
where $(N^{n-1}(0),\bar g)$ has flat scalar curvature, and the function $\psi$ satisfies $\psi''(r)>0$ and $\psi(r)\rightarrow0$ as $r\nearrow +\infty$, or

\item
the warped product 
\[
(\mathbb{R},dr^2)\times_{\psi}(N^{n-1}(-c^2),\bar g),
\]
where $(N^{n-1}(-c^2),\bar g)$ has negative constant scalar curvature, and the function $\psi$ satisfies $\psi''(r)>0$ and $\psi(r)\rightarrow\sqrt{\frac{\bar R}{\lambda}}$ for $\bar R<0$ as $r\nearrow +\infty$.

\end{enumerate}
\end{proposition}

\begin{proof}
By the assumption, we have $\psi'<0$. If $\psi''<0$ on $\mathbb{R}$ (resp. $(0,+\infty)$), then since $\psi>0$, we have a contradiction.
Assume that there exists $\tilde r\in\mathbb{R}$ (resp. $\tilde r\in(0,+\infty)$) such that $\psi''=0$. By \eqref{R3},  we have 
\[
\psi'''
=-\frac{\psi'}{2(n-1)}
\{
2\lambda+2\psi'e^{cF}+c\psi^2e^{2cF}
\}
<0.
\]
By iterating the same argument, $\psi'$ is weakly monotone decreasing on $(\tilde r,+\infty)$.
Since $\psi'<0$, we have 
$\psi'<-\gamma^2$ for some constant $\gamma$ on an interval $(\tilde r_1,+\infty)$.
Since $\psi>0$, we have a contradiction.
Hence, we have $\psi''>0$ on $\mathbb{R}$ (resp. $(0,+\infty)$). We have
$\psi\searrow\alpha$ for some $\alpha\geq0$ and $\psi'\nearrow0$ as $r\nearrow+\infty.$ 

Case I $\alpha>0$: Since $\psi''\rightarrow \frac{\bar R-\lambda\alpha^2}{2(n-1)\alpha}$ as $r\nearrow+\infty$.
We consider three cases. 

Case 1 $\bar R>\lambda\alpha^2$: 
In this case, $\psi''\rightarrow \frac{\bar R-\lambda\alpha^2}{2(n-1)\alpha}>0$ as $r\nearrow+\infty$. Since $\psi'<0$, we have a contradiction.

Case 2 $\bar R=\lambda\alpha^2$: 
In this case, $\psi''\rightarrow 0$ as $r\nearrow+\infty$. 

Case 3 $\bar R<\lambda\alpha^2$:
In this case, $\psi''\rightarrow \frac{\bar R-\lambda\alpha^2}{2(n-1)\alpha}<0$ as $r\nearrow+\infty$. Since $\psi>0$, we have a contradiction.

Case II $\alpha=0$: 

Case 1 $\bar R>0$: By \eqref{R2}, we have $\psi''\rightarrow+\infty$ as $r\nearrow+\infty.$ Since $\psi'<0$, we have a contradiction.

Case 2 $\bar R<0$: 
By \eqref{R2}, we have $\psi''\rightarrow-\infty$ as $r\nearrow+\infty.$ Since $\psi>0$, we have a contradiction.

Therefore, we have $\bar R=0$.
\end{proof}

By a similar argument, we classify complete steady quasi-Yamabe gradient solitons.

\begin{proposition}\label{ste,r<l,c>0}
Let $(M,g,F,c)$ be an $n$-dimensional $(n\geq3)$ complete nontrivial steady quasi-Yamabe gradient soliton with $R<0$ and $c<0$, then $M$ is 
the warped product 
\[
(\mathbb{R},dr^2)\times_{\psi}(N^{n-1}(0),\bar g),
\]
where $(N^{n-1}(0),\bar g)$ has flat scalar curvature, and the function $\psi$ satisfies $\psi''(r)>0$ and $\psi(r)\rightarrow0$ as $r\nearrow +\infty$.
\end{proposition}

Finally, we provide a list of unsolved problems for complete quasi-Yamabe gradient solitons with bounded scalar curvature (see also \cite{Maeta25}).

\begin{table}[htbp]
  \centering
  \begin{tabular}{|l|l|c|c|c|}
    \hline
    Condition on $R$ & Condition on $c$ & Shrinking & Steady & Expanding\\
    \hline\hline
    \multirow{2}{*}{$R>\lambda+\varepsilon$}
      & $c>0$ & Proposition~\ref{r+v} & Proposition~\ref{r+v} & Proposition~\ref{r+v}\\ \cline{2-5}
      & $c<0$ & \textit{Unsolved}     & \textit{Unsolved}     & \textit{Unsolved}\\
    \hline
    \multirow{2}{*}{$R>\lambda$}
      & $c>0$ & Theorem~\ref{nnp} & Theorem~\ref{nnp} & \textit{Unsolved}\\ \cline{2-5}
      & $c<0$ & \textit{Unsolved} & \textit{Unsolved} & \textit{Unsolved}\\
    \hline
    \multirow{2}{*}{$R<\lambda$}
      & $c>0$ & \textit{Unsolved}           & \textit{Unsolved}           & \textit{Unsolved}\\ \cline{2-5}
      & $c<0$ & \textit{Unsolved} & Proposition~\ref{ste,r<l,c>0} & Proposition~\ref{e,r<l,c<0}\\
    \hline
    \multirow{2}{*}{$R<\lambda-\varepsilon$}
      & $c>0$ & \textit{Unsolved}     & \textit{Unsolved}     & \textit{Unsolved}\\ \cline{2-5}
      & $c<0$ & Proposition~\ref{r-v} & Proposition~\ref{r-v} & Proposition~\ref{r-v}\\
    \hline
  \end{tabular}
  \caption{Summary of results for the Shrinking, Steady, and Expanding quasi-Yamabe gradient solitons.}
\end{table}
~\\







~\\
\bibliographystyle{amsbook}

\end{document}